\documentclass[a4paper,12pt]{amsart}

\usepackage{mathrsfs}
\usepackage{amsfonts,amssymb,amscd,amsthm,amsmath,graphicx}
\usepackage[all]{xy}
\usepackage{color}

\theoremstyle{plain}
\newtheorem{theorem}{Theorem}[section]
\newtheorem{lemma}[theorem]{Lemma}

\newtheorem{proposition}[theorem]{Proposition}

\newtheorem{definition}{Definition}[section]

\newtheorem{example}{Example}[section]

\def\G{\operatorname{G}}
\def\GL{\operatorname{GL}}

\def\U{\operatorname{U}}

\def\ad{\operatorname{ad}}

\def\diag{\operatorname{diag}}

\def\exp{\operatorname{exp}}

\def\ker{\operatorname{ker}}

\def\max{\operatorname{max}}

\def\pr{\operatorname{pr}}

\def\Stab{\operatorname{Stab}}

\def\tr{\operatorname{tr}}

\newcommand{\frg}{\mathfrak{g}}

\begin{document}

\title{A geometric interpretation of Kirillov's conjecture}
\date{June 2018}
\thanks{The authors thank David Vogan for suggesting this problem.}
\author{Gang Liu and Jun Yu}

\address[Gang Liu]{Universit\'{e} de Lorraine,
Institut Elie Cartan de Lorraine,  3 rue Augustin Fresnel,
57073 Metz, France.}
\email{gang.liu@univ-lorraine.fr}

\address[Jun Yu]{BICMR, Peking University, No. 5 Yiheyuan Road, Haidian District, Beijing 100871, China.}
\email{junyu@bicmr.pku.edu.cn}

\abstract{In this article, we give a natural geometric interpretation of Kirillov's conjecture for tempered
representations in the framework of orbit method. This interpretation can also be considered
as a first generalization of Duflo's conjecture to non-discrete series representations.}
\endabstract

\maketitle



\tableofcontents

\section{Introduction}

Let $G=\GL(n, \mathbb{K})$, where $\mathbb{K}=$ $\mathbb{R}$ or $\mathbb{C}$, and let $P$ be
the subgroup of  $G$ with matrices whose last row is $(0,0,\dots,0,1)$. Let $\pi$ be an
irreducible unitary representation of $G$. Then we have the following theorem known as
"Kirillov's conjeture":

\begin{theorem}\label{Kirillov conjecture}
The restriction of $\pi$ to $P$, $\pi\vert_{P}$ is irreducible.
\end {theorem}

Kirillov's conjecture has a quite long story,  among the main contributors of its proof,
are Sahi, Sahi-Stein and  Baruch (who finally proved the conjecture in all generality in \cite{Baruch}).
It is noted that not only the statement of Kirillov's conjecture is of purely analytic
nature, the final proof of Baruch is also essentially analytic.

Kirillov's conjecture is of course a branching problem. Concerning branching problems for
reductive groups, significant progress has been made during last twenty years, notably by
Kobayashi (and his collaborators). However in these works, subgroups are almost always
reductive as well.

On the other hand, by works of Kostant, Kirillov, Souriau, Duflo and others, we know that
geometric methods have played an important role in the development of representation theory
and harmonic analysis. Then we may ask,  despite the analytic nature of branching problems,
can we interpret the branching laws in some geometric manner? Or can branching laws be
essentially determined by some geometric data (e.g. coadjoint orbits) related to
representations? The answer to this problem is positive in some context, more precisely,
in some cases, we can describe the branching laws in the framework of \emph{orbit method},
or more generally of \emph{geometric quantization} . In this direction, the first success
concerned the case where the group $G$ is compact as initiated in the thesis of Heckman
in the framework of orbit method, and generalized by  the Guillemin-Sternberg conjecture
(proved by Meinrenken and also Tian-Zhang)  in the framework of geometric quantization
under the celebrated slogan  \emph{quantization commutes with reduction}. Since then,
along this direction, the theory was quite generalized developed notably by Paradan,
Vergne, also Ma and Zhang (especially the proof of Vergne's conjecture). Nevertheless,
it should be pointed out that in almost all these works, the groups and subgroups are
reductive. Encouraged by all these developments, more recently, Duflo formulated a
conjecture which aims at giving a geometric interpretation for the branching problem
$(G,H,\pi)$, where $G$ is an (almost) algebraic group, $H$ is an (almost) algebraic
subgroup of $G$, and  $\pi$ is a  \emph{discrete series representation} of $G$.
For more details concerning Duflo's conjecture, see the next section.

The main purpose of this article is to give a geometric interpretation of Kirillov's
conjecture for \emph{tempered representations} in the spirit of Duflo's conjecture.
Since Duflo's conjecture in its initial formulation concerns only discrete series
representations, and in our case (namely $G=\GL(n, \mathbb{K})$), there are no any
discrete series representations, our work can also be considered as a (and to the
best of our knowledge, a first) generalization of Duflo's conjecture.  




\section{Duflo's conjecture} \

In this section, we will state Duflo's initial conjecture. For the reader's sake, we will begin by outlining some essential ingredients in Duflo's orbit method. For more details and the general setting about the theory of Duflo's orbit method, we refer to Duflo's "CIME lectures" \cite{Duflo1}.

Let $G$ be an almost algebraic real group with Lie algebra $\mathfrak{g}$.  Denote by $\mathfrak{g}^*$ the algebraic dual of $\mathfrak{g}$.  In the framework of Duflo's theory,  a fundamental  notion is \emph{admissible} (in the sense of Duflo) and   \emph{well polarizable}  (in the sense of Pukanszky) $G$-coadjoint orbits
(in $\mathfrak{g}^*$).  Each such orbit $\mathcal{O}$ is attached to at least one irreducible unitary representation of $G$. Moreover, two different such orbits are associated to non-equivalent representations.  In general, the set of irreducible unitary $G$-representation  associated to admissible and well polarizable $G$-codajoint orbits is not the whole unitary dual $\hat{G} $. However, it is sufficient to describe the Plancherel formula of $G$ (namely, the decomposition of $L^2(G)$) by Duflo's theory. Especially all \emph{discrete series} representations of $G$ (i.e., those appearing in the discrete part of the Plancherel formula of $G$), and "almost" all  \emph{ tempered } representations of $G$  (i.e., those appearing in the spectral decomposition of $L^2(G)$) are attached to admissible \emph{stronlgy regular} $G$-adjoint orbits.  Recall that an element $f\in \mathfrak{g}^*$ is called strongly reagular, if $f$ is regular (i.e., the coadjoint orbit of $f$ is of maximal dimension) and its "reductive factor"
$\mathfrak{s}(f):= \{X\in \mathfrak{g}(f):\text{ad} X \  \text{is semisimple}  \}$ is  of maximal dimension among the reductive factors of all the
regular elements in  $\mathfrak{g}^*$.  A coadjoint orbit $\mathcal{O}$ is called strongly regular, if there exists
an element $f\in \mathcal{O}$ (then each $f\in \mathcal{O}$ ) which is strongly regular. Notice that each strongly regular coadjoint orbit is automatically well polarizable.

Now let $H$ be an almost algebraic subgroup of $G$ with Lie algebra $\mathfrak{h}$.  Let $\mathcal{O}$ be a $G$-cadjoint orbit (in $\mathfrak{g}^*$).  It is well known that equipped with the Kirillov-Kostant-Souriau symplectic form $\mathcal{ \omega}$, $\mathcal{O}$ becomes a $H$-Hamiltonian space. The corresponding moment map is just the natural projection $\text{p}: \mathcal{O} \rightarrow\mathfrak{h}^{\ast}$

Let $\pi$ be a discrete series of $G$, as stated above, it is attached to a strongly regular $G$-coadjoint orbit $\mathcal{O}_{\pi}$.  Consider the restriction of $\pi$ to $H$, $\pi\vert_{H}$. Then in the context, Duflo's conjecture states as follows:\\

\begin{itemize}

  \item[i)] $\pi\vert_{H}$ is $H$-admissible (in the sense of Kobayashi) if and only if the moment map $\text{p}: \mathcal{O}_{\pi}\rightarrow\mathfrak{h}^{\ast}$ is \textit{weakly proper}.

 \item[ii)]  If $\pi\vert_{H}$ is $H$-admissible, then each irreducible $H$-representation $\sigma$ which appears in $\pi\vert_{H}$ is attached to a \textit{strongly regular} $H$-coadjoint orbit $\Omega$ (in the sense of Duflo) which is contained in $\text{p}(\mathcal{O}_{\pi})$.

\item[iii)]  If $\pi\vert_{H}$ is $H$-admissible, the multiplicity of each such $\sigma$  can be expressed geometrically on the \textit{reduced space}  of $\Omega$ (with respect to the moment map $\text{p}$).

 \end{itemize}

Let us give some more explanations for Duflo's conjecture. Firstly, the notion "$H$-admissible" above is due to Kobayashi, which means that $\pi\vert_{H}$ decomposes discretely and with finite multiplicities. 

 The "weak properness" in i) means that the preimage (for $\text{p}$) of each compact subset which is contained in $\text{p}(\mathcal{O}_{\pi})\cap\Upsilon_{sr}$ is compact in $\mathcal{O}_{\pi}$. Here $\Upsilon_{sr}$ is the set of all strongly regular elements in  $\mathfrak{h}^*$. 

  For (ii), as we already mentioned above,  each discrete series of $G$ (resp. $H$) is attached to a strongly regular $G$ (resp. $H$)-coadjoint orbit. Moreover according to Duflo-Vargas's work (\cite{Du-Va1}, \cite{Du-Va2}), each irreducible $H$-representation $\widetilde{\sigma}$ which appears in the integral decomposition of  $\pi\vert_{H}$ (which is not necessarily $H$-admissible) is attached to a strongly regular $H$-coadjoint orbit. Note that $\widetilde{\sigma}$ is not necessarily a discrete series. However, if $\pi\vert_{H}$ is $H$-admissible, then each $H$-irreducible representation appearing in $\pi\vert_{H}$ must be a discrete series. Thus  (ii) has a nice geometric meaning.

  Despite some progress (see for example \cite{Liu}), Duflo's conjecture is still not fully established. Nevertheless, it seems reasonable  to ask if we can generalize Duflo's conjecture (possibly with "adapted" modifications) to larger family of unitary irreducible representations of $G$ ( which are not necessarily discrete series). A first such attempt could be naturally for tempered representations, since tempered representations are closely related to discrete series, and almost all of them are attached to strongly regular coadjoint orbits.  In fact, we will prove the generalization in this direction under the setting of Kirillov's conjecture, which, in return, gives a geometric interpretation of Kirillov's conjecture for tempered representations.

\section{Geometry of the moment map $p:\mathcal{O}_{f}:\rightarrow\mathfrak{p}^{\ast}$}\label{S:moment}

\subsection{Coadjoint action, the dual map and the moment map}\label{SS:Porbit}

Let $n\geq 1$, and $k=\mathbb{R}$ or $\mathbb{C}$. Write $G_{n}(k)=\GL(n,k)$, $$P_{n}(k)=
\{\left( \begin{array}{cc}A&\alpha\\0_{1\times(n-1)}&1\\\end{array}\right):A\in\GL(n-1,k),
\alpha\in k^{n-1}\}.$$ In the literature $P_{n}(k)$ is called a microbolic subgroup. Write
$\mathfrak{g}_{n}(k)=\mathfrak{gl}(n,k)$, which is the algebra of $G_{n}(k)$. Write
$$\mathfrak{p}_{n}(k)=\{\left(\begin{array}{cc}A&\alpha\\0_{1\times(n-1)}&0\\\end{array}
\right):A\in\mathfrak{gl}(n-1,k),\alpha\in k^{n-1}\},$$ which is the Lie algebra of $P_{n}(k)$.

Write $\mathfrak{g}_{n}(k)^{\ast}$ (or $\mathfrak{p}_{n}(k)^{\ast}$) for the dual space of
$\mathfrak{g}_{n}(k)$ (or $\mathfrak{p}_{n}(k)$). Then, $G=G_{n}(k)$ (or $P_{n}(k)$) acts on
$\frg^{\ast}=\mathfrak{g}_{n}(k)^{\ast}$ (or $\mathfrak{p}_{n}(k)^{\ast}$) through $$(g\cdot f)\xi
=f(g^{-1}\cdot\xi),\ \forall g\in G,\forall f\in\frg^{\ast},\forall\xi\in\frg.$$ This is called 
the {\it coadjoint action}, and a $G$ orbit in $\frg^{\ast}$ is called a coadjoint orbit.

Write $$(\xi,\eta)=\tr(\xi\eta),\ \forall\xi,\eta\in\mathfrak{g}_{n}(k).$$ This gives a $G_{n}(k)$
conjugation invariant nondegenerate bilinear form on $\mathfrak{g}_{n}(k)$. It gives a $G_{n}(k)$ 
equivariant isomorphism $$\pr:\mathfrak{g}_{n}(k)\rightarrow\mathfrak{g}_{n}(k)^{\ast},\quad\xi
\mapsto f$$ defined by $$f(\eta)=\tr(\xi\eta),\ \forall\eta\in\mathfrak{g}_{n}(k).$$ Through $\pr$, 
the above bilinear form on $\mathfrak{g}_{n}(k)$ induces a nondegenerate bilinear form on 
$\mathfrak{g}_{n}(k)^{\ast}$ defined by $$(f,g)=(\pr^{-1}(f),\pr^{-1}(g)),\forall f,g\in
\mathfrak{g}_{n}(k)^{\ast}.$$

Write $$\overline{\mathfrak{p}}_{n}(k)=\{\left(\begin{array}{cc}A&0_{(n-1)\times 1}\\\alpha^{t}&0
\\\end{array}\right):A\in\mathfrak{gl}(n-1,k),\alpha\in k^{n-1}\}.$$ Define $\pr':\mathfrak{g}_{n}(k)
\rightarrow\mathfrak{p}_{n}(k)^{\ast}$ by $$(\pr'(\xi))(\eta)=\tr(\xi\eta),\ \forall\eta\in
\mathfrak{p}_{n}(k).$$ It is easy to show that $$\ker(\pr')=\{\left(\begin{array}{cc}
0_{(n-1)\times(n-1)}&\alpha\\0_{1\times(n-1)}&t\\\end{array}\right):\alpha\in k^{n-1},t\in k.\}$$
It is clear that $\mathfrak{g}_{n}(k)=\ker(\pr')\oplus\overline{\mathfrak{p}}_{n}(k)$ as a linear
space. Thus, $\pr'|_{\overline{\mathfrak{p}}_{n}(k)}:\overline{\mathfrak{p}}_{n}(k)\rightarrow
\mathfrak{p}_{n}(k)^{\ast}$ is a linear isomorphism. In this way, any element in $f\in
\mathfrak{p}_{n}(k)^{\ast}$ could be represented by $$f=\pr'(\xi)=\pr(\xi)|_{\mathfrak{p}_{n}(k)}$$
for a unique $\xi\in\overline{\mathfrak{p}}_{n}(k)$.

The {\it moment map} $p:\mathfrak{g}_{n}(k)^{\ast}\rightarrow\mathfrak{p}_{n}(k)^{\ast}$ is defined
by $$f\mapsto f'=f|_{\mathfrak{p}_{n}(k)}.$$ For any $\xi\in\mathfrak{g}_{n}(k)$, we have
$$p(\pr(\xi))=\pr'(\xi).$$

Note also that for any Lie group $G$, the coadjoint action of $\mathfrak{g}$ on $\mathfrak{g}^{\ast}$
(which is the differential of the coadjoint action of $G$ on $\mathfrak{g}^{\ast}$) is determined by
$$(\ad(X)f)(Y)=-f([X,Y]),\ \forall X,Y\in\mathfrak{g},\forall f\in\mathfrak{g}^{\ast}.$$

\subsection{The classification of $P$ coadjoint orbits}

Write $$N_{n}(k)=\{\left(\begin{array}{cc}I_{n-1}&\alpha\\0_{1\times(n-1)}&1\\\end{array}\right):
\alpha\in k^{n-1}\},$$ and $$\mathfrak{n}_{n}(k)=\{\left(\begin{array}{cc}0_{n-1}&\alpha
\\0_{1\times(n-1)}&0\\\end{array}\right):\alpha\in k^{n-1}\}.$$ Then, $N_{n}(k)$ is the unipotent
radical of $P_{n}(k)$, and $\mathfrak{n}_{n}(k)$ is its Lie algebra (=nilpotent radical of
$\mathfrak{p}_{n}(k)$). Write $$L_{n}(k)=\{\left(\begin{array}{cc}A&0_{(n-1)\times 1}\\0_{1\times(n-1)}
&1\\\end{array}\right):A\in G_{n-1}(k)\},$$ and $$\mathfrak{l}_{n}(k)=\{\left(\begin{array}{cc}A&
0_{(n-1)\times 1}\\0_{1\times(n-1)}&0\\\end{array}\right):A\in\mathfrak{g}_{n-1}(k).\}.$$ Then,
$L_{n}(k)$ (or $\mathfrak{l}_{n}(k)$) is a Levi subgroup (or Levi subalgebra) of $P_{n}(k)$ (or of
$\mathfrak{l}_{n}(k)$).

We could identify $\mathfrak{l}_{n}(k)$ with $\mathfrak{p}_{n}(k)/\mathfrak{n}_{n}(k)$. By this, there is
an exact sequence of $P_{n}(k)$ modules, $$0\rightarrow\mathfrak{n}_{n}(k)\rightarrow\mathfrak{p}_{n}(k)
\rightarrow\mathfrak{l}_{n}(k)\rightarrow 0.$$ Dually, there is an exact sequence of $P_{n}(k)$ modules,
$$0\rightarrow\mathfrak{l}_{n}(k)^{\ast}\rightarrow\mathfrak{p}_{n}(k)^{\ast}\rightarrow
\mathfrak{n}_{n}(k)^{\ast}\rightarrow 0.$$

\begin{lemma}\label{L:Porbit-inductive}
For any $n\geq 1$, there is an identification $$\mathfrak{p}_{n}(k)^{\ast}/P_{n}(k)=
\mathfrak{l}_{n}(k)^{\ast}/L_{n}(k)\bigsqcup\mathfrak{p}_{n-1}(k)^{\ast}/P_{n-1}(k).$$
\end{lemma}
\begin{proof}
From the exact sequence $0\rightarrow\mathfrak{l}_n(k)^{\ast}\rightarrow\mathfrak{p}_n(k)^{\ast}
\rightarrow\mathfrak{n}_n(k)^{\ast}\rightarrow 0$, we get $$\mathfrak{p}_{n}(k)^{\ast}/P_{n}(k)=
\mathfrak{l}_{n}(k)^{\ast}/P_n(k)\bigsqcup(\mathfrak{p}_n(k)^{\ast}-
\mathfrak{l}_n(k)^{\ast})/P_{n}(k).$$ As $N_{n}$ acts trivially on $\mathfrak{l}_{n}(k)^{\ast}$,
we get $$\mathfrak{l}_{n}(k)^{\ast}/P_n(k)=\mathfrak{l}_{n}(k)^{\ast}/L_n(k).$$

Choose an element $\overline{h}\in\mathfrak{p}_{n}(k)^{\ast}$ such that $0\neq h=
\overline{h}|_{\mathfrak{n}_{n}(k)^{\ast}}\in\mathfrak{n}_{n}(k)^{\ast}$ and
$\overline{h}|_{\mathfrak{l}_{n}(k)}=0$. Due to the fact that $L_{n}(k)$ acts transitively on 
$\mathfrak{n}_{n}(k)^{\ast}-\{0\}$, any $P_{n}(k)$ orbit in $\mathfrak{p}_n(k)^{\ast}-
\mathfrak{l}_n(k)^{\ast}$ intersects with $\mathfrak{l}_{n}(k)^{\ast}+\overline{h}$. From this we
get $$(\mathfrak{p}_{n}(k)^{\ast}-\mathfrak{l}_{n}(k)^{\ast})/P_n(k)\cong(\mathfrak{l}_{n}(k)^{\ast}
+\overline{h})/P_n(k)^{h},$$ where $P_{n}(k)^{h}=\Stab_{P_n(k)}(h)$. Moreover we write $L_{n}(k)^{h}
=\Stab_{L_n(k)}(h)$, and write $\mathfrak{p}_{n}(k)^{h}$ (or $\mathfrak{l}_{n}(k)^{h}$) for the Lie 
algebra of $P_{n}(k)^{h}$ (or of $L_{n}(k)^{h}$). As $N_{n}(k)$ is abelian, we have $N_{n}(k)
\subset P_n(k)^{h}$. Thus, $$P_n(k)^{h}=N_{n}(k)\rtimes L_n(k)^{h}.$$  

Since $N_{n}(k)$ acts trivially on $\mathfrak{l}_{n}(k)^{\ast}$, its action on 
$\mathfrak{l}_{n}(k)^{\ast}+\overline{h}$ is through translations. Actually, $$\exp(X)\cdot(g+
\overline{h})=g+\overline{h}+\ad(X)\overline{h}$$ for any $X\in\mathfrak{n}_{n}(k)$ and any $g\in
\mathfrak{l}_{n}(k)^{\ast}.$ In the below we show that \begin{equation}\label{Eq1}
\mathfrak{l}_{n}(k)^{\ast}/\ad(\mathfrak{n}_{n}(k))\overline{h}=(\mathfrak{l}_{n}(k)^{h})^{\ast}.
\end{equation} From this it follows that $$(\mathfrak{l}_n(k)^{\ast}+ h)/P_n(k)^{h}\cong
\mathfrak{l}_{n}(k)^{h}/L_{n}(k)^{h}.$$ It is easy to check directly that $L_n(k)^{h}\cong P_{n-1}
(k)$. Then it follows that $$(\mathfrak{p}_{n}(k)^{\ast}-\mathfrak{l}_{n}(k)^{\ast})/P_{n}(k)\cong
\mathfrak{p}_{n-1}(k)^{\ast}/P_{n-1}(k).$$ Therefore, $$\mathfrak{p}_{n}(k)^{\ast}/P_{n}(k)=
\mathfrak{l}_{n}(k)^{\ast}/L_{n}(k)\bigsqcup\mathfrak{p}_{n-1}(k)^{\ast}/P_{n-1}(k).$$

Now we show $\mathfrak{l}_{n}(k)^{\ast}/\ad(\mathfrak{n}_{n}(k))\overline{h}=(\mathfrak{l}_{n}
(k)^{h})^{\ast}.$ Since the pairing between $\mathfrak{l}_{n}(k)$ and $\mathfrak{l}_{n}(k)^{\ast}$ is
nondegenerate. It is equivalent to show \begin{equation}\label{Eq2}\{\xi\in\mathfrak{l}_{n}(k):
(\ad(\eta)\overline{h})(\xi)=0,\forall\eta\in\mathfrak{n}_{n}(k)\}=\mathfrak{l}_{n}(k)^{h}.
\end{equation} This follows from $$(\ad(\eta)\overline{h})(\xi)=-\overline{h}([\eta,\xi])=
\overline{h}([\xi,\eta])=-(\ad(\xi)\overline{h})(\eta).$$
\end{proof}

\begin{lemma}\label{L:Porbit-invariant}
Assume $\overline{h}|_{\mathfrak{l}_{n}(k)^{\ast}}\!=\!0$ and $0\!\neq\!h\!=\!
\overline{h}|_{\mathfrak{n}_{n}(k)^{\ast}}\!\in\!\mathfrak{n}_{n}(k)^{\ast}$. Then, \begin{itemize}
\item[(1),] $\ad(\mathfrak{l}_{n}(k)^{h})\overline{h}=0$.
\item[(2),] the map $$\mathfrak{n}_{n}(k)\rightarrow\mathfrak{l}_{n}(k)^{\ast},\xi\mapsto
(\ad\xi)\overline{h}$$ is injective.
\end{itemize}
\end{lemma}

\begin{proof}
For assertion (1), let $\xi\in\mathfrak{l}_{n}(k)^{h}$. For any $\eta\in\mathfrak{n}_{n}(k)$,
$(\ad(\xi)\overline{h})(\eta)=0$ since $\ad(\xi)\overline{h}\in\mathfrak{l}_{n}(k)^{\ast}$. For any
$\eta\in\mathfrak{l}_{n}(k)$, $$(\ad(\xi)\overline{h})(\eta)=-\overline{h}([\xi,\eta])=0$$ since
$\overline{h}|_{\mathfrak{l}_{n}(k)^{\ast}}=0$. Thus, $\ad(\xi)\overline{h}=0.$ Therefore,
$\ad(\mathfrak{l}_{n}(k)^{h})\overline{h}=0$.

For assertion (2), we may assume that $\overline{h}=\pr'(\xi)$, where $\xi'=\left(\begin{array}{cc}
0_{n-1}&0_{(n-1)\times 1}\\\beta'^{t}&0\end{array}\right)$, $\beta'^{t}=(0,\dots,0,1)$. Then, a
direct calculation of $\ad(\eta)\overline{h}=\pr'([\eta,\xi'])$ ($\eta\in\mathfrak{n}_{n}(k)$)
shows the assertion.
\end{proof}

Now assume $\overline{h}|_{\mathfrak{l}_{n}(k)^{\ast}}=0$ and $0\neq h=
\overline{h}|_{\mathfrak{n}_{n}(k)}\in\mathfrak{n}_{n}(k)^{\ast}$. Choose a complement of
$(\ad(\mathfrak{n}_{n}(k))\overline{h}$ in $\mathfrak{l}_{n}(k)^{\ast}$, denoted by $V_{h}$.

\begin{lemma}\label{L:Porbit-stabilizer}
Assume $\overline{h}|_{\mathfrak{l}_{n}(k)^{\ast}}\!=\!0$ and $0\!\neq\!h\!=\!
\overline{h}|_{\mathfrak{n}_{n}(k)^{\ast}}\!\in\!\mathfrak{n}_{n}(k)^{\ast}$. Then, we have the following 
assertions, \begin{itemize} 
\item[(1),]each $P_{n}(k)$ orbit intersecting with $\mathfrak{l}_{n}(k)^{\ast}+\overline{h}$ has a 
representative of the form $f=g+\overline{h}$ where $g\in V_{h}.$
\item[(2),] Two elements $g_1+\overline{h}$ and $g_2+\overline{h}$ ($g_1,g_2\in V_{h}$) are in one $P_{n}(k)$
orbit if and only if $[g_1]$ and $[g_2]$ are in one $L_{n}(k)^{h}$ orbit, where$$[g_{i}]\!=
\!g_{i}\!+\!\ad(\mathfrak{n}_{n}(k))\overline{h}\in\mathfrak{l}_{n}(k)^{\ast}/\ad(\mathfrak{n}_{n}(k))
\overline{h}\!=\!(\mathfrak{l}_{n}(k)^{\overline{h}})^{\ast}$$ ($i=1,2$) are considered as elements in
$(\mathfrak{l}_{n}(k)^{h})^{\ast}$.
\item[(3),] Assume $g\in V_{h}$. Then, the map $$\Stab_{P_{n}(k)}(g+\overline{h})\rightarrow L_{n}(k),\
nl\mapsto l$$ ($n\in N_{n}(k)$, $l\in L_{n}(k)$) gives an isomorphism $$\Stab_{P_{n}(k)}(g+\overline{h})
\cong\Stab_{L_{n}(k)^{h}}([g]).$$
\end{itemize}
\end{lemma}

\begin{proof}
The assertion $(1)$ follows from Equation $(1)$.

For assertion $(2)$, we show the necessarity first. Assume $g_2+\overline{h}=x\cdot(g_1+\overline{h})$
for some $x\in P_{n}(k)$. Then, $x\in P_{n}(k)^{h}=N_{n}(k)\rtimes L_{n}(k)^{h}$. Write $x=nl=
\exp(\xi)l$ for some $\xi\in\mathfrak{n}_{n}(k)$ and $l\in L_{n}(k)^{h}$. Then, $$\exp(-\xi)\cdot
(g_2+\overline{h})=l\cdot(g_1+\overline{h}).$$ By Lemma \ref{L:Porbit-invariant}, we have $l\cdot
\overline{h}=0$. As $\mathfrak{n}_{n}(k)$ is an abelian ideal of $\mathfrak{n}_{n}(k)$, we have
$\exp(-\xi)\cdot(g_2+\overline{h})=(g_2-\ad(\xi)\overline{h})+\overline{h}$. Thus, $g_2-\ad(\xi)
\overline{h}=l\cdot g_1$. This just means, $[g_1]$ and $[g_2]$ are in one $L_{n}(k)^{h}$ orbit. The
sufficiency could be shown with similar facts used in showing the necessarity.

For assertion $(3)$, write $n=\exp(\xi)$ ($\xi\in\mathfrak{n}_{n}(k)$). By the above proof for (2),
we see that $nl\in\Stab_{P_{n}(k)}(g+\overline{h})$ if and only if $l\in L_{n}(k)^{h}$ and $$g-
\ad(\xi)\overline{h}=l\cdot g.$$ The last is just the condition for $l\in\Stab_{L_{n}(k)^{h}}([g])$.
Thus, the map $$\Stab_{P_{n}(k)}(g+\overline{h})\rightarrow L_{n}(k),\ nl\mapsto l$$ gives a
surjection $\Stab_{P_{n}(k)}(g+\overline{h})\rightarrow\Stab_{L_{n}(k)^{h}}([g]).$ On the other
hand, suppose $l=1$. Then, $\ad(\xi)\overline{h}=0$. By Lemma \ref{L:Porbit-invariant}(2), this
implies that $\xi=0$. Thus, the above surjection $\Stab_{P_{n}(k)}(g+\overline{h})\rightarrow
\Stab_{L_{n}(k)^{h}}([g])$ is an isomorphism.
\end{proof}

We have some remarks regarding Lemma \ref{L:Porbit-stabilizer}(3). \begin{itemize}
\item[(1),] the statement is valid for any $g\in\mathfrak{l}_{n}(k)^{\ast}$. 
\item[(2),] $(\ad(\mathfrak{n}_{n}(k))\overline{h}$ is stable under the action of $L_{n}(k)^{h}$, but 
$L_{n}(k)^{h}$ is not a reductive subgroup. Actually one can verify that we couldn't take $V_{h}$ stable 
under $L_{n}(k)^{h}$. 
\item[(3),] we have $$\Stab_{L_{n}(k)^{h}}(g)\subset\Stab_{P_{n}(k)}(g+\overline{h})$$ and 
$$\Stab_{L_{n}(k)^{h}}(g)\subset\Stab_{L_{n}(k)^{h}}([g]),$$ but neither is an equality in general.
\end{itemize}

The following lemma is easy to show. 
\begin{lemma}\label{L:Porbit-stabilizer2}
When $h=0$, we have $f\in\mathfrak{l}_{n}(k)^{\ast}\subset\mathfrak{n}_{n}(k)^{\ast}$. Then, 
$$\Stab_{P_{n}(k)}(f)=N_{n}(k)\rtimes\Stab_{L_{n}(k)}(f).$$ 
\end{lemma}

\begin{theorem}\label{T:Porbit}
Any $P_{n}(k)$ coadjoint orbit $P_{n}(k)\cdot f$ is characterized by an integer $j$ ($0\leq j\leq n-1$)
and an $L_{n-j}(k)$ coadjoint orbit $L_{n-j}(k)\cdot g$. Moreover, we have $\Stab_{P_{n}(k)}f\cong
N_{n-j}(k)\rtimes\Stab_{L_{n-j}(k)}(g)$.
\end{theorem}

\begin{proof}
Applying Lemma \ref{L:Porbit-inductive} inductively, we get $$\mathfrak{p}_{n}(k)^{\ast}/P_{n}(k)=
(\bigsqcup_{0\leq j\leq n-2}\mathfrak{l}_{n-j}(k)^{\ast}/L_{n-j}(k))\bigsqcup\mathfrak{p}_{1}
(k)^{\ast}/P_{1}(k).$$ As both $P_{1}(k)$ and $L_{1}(k)$ are the trivial group, we get $\mathfrak{p}_{1}
(k)^{\ast}/P_{1}(k)=\mathfrak{l}_{1}(k)^{\ast}/L_{1}(k)=\textrm{singleton}.$ Thus, \begin{equation}
\label{Eq:depth}\mathfrak{p}_{n}(k)^{\ast}/P_{n}(k)=\bigsqcup_{0\leq j\leq n-1}\mathfrak{l}_{n-j}
(k)^{\ast}/L_{n-j}(k).\end{equation} For any $j$, $L_{n-j}(k)\cong\G_{n-1-j}(k)$. Thus, we get the
first statement of the theorem.

If a $P_{n}(k)$ coadjoint orbit $P_{n}(k)\cdot f$ is characterized by an integer $j$ ($0\leq j\leq n-1$) 
and an $L_{n-j}(k)$ coadjoint orbit $L_{n-j}\cdot g$, then $\Stab_{P_{n}(k)}f\cong\Stab_{L_{n-j}}(g)$ by 
Lemma \ref{L:Porbit-stabilizer}(3) and Lemma \ref{L:Porbit-stabilizer2}.
\end{proof}

Note that $N_{n-j}(k)\cong k^{n-1-j}$. With Theorem \ref{T:Porbit}, we not only classified all $P$ coadjoint 
orbits, but also calculated their {\it stabilizers}. 

\begin{definition}\label{D:depth}
According to Equation (\ref{Eq:depth}), any $P_{n}(k)$-coadjoint orbit $P_{n}(k)\cdot f\subset
\mathfrak{p}_{n}(k)^{\ast}$ is parameterized by an integer $j$ ($0\leq j\leq n-1$) and a $L_{n-j}(k)$ 
coadjoint orbit $L_{n-j}(k)\cdot g\subset \mathfrak{l}_{n-j}(k)^{\ast}$. We call $j+1$ the {\it depth} 
of this orbit (and points in it). 
\end{definition}

By Theorem \ref{T:Porbit}, $\Stab_{P_{n}(k)}(f)\cong N_{n-j}(k)\rtimes\Stab_{L_{n-j}(k)}(g)$. We call $f$ 
a {\it semisimple element} (or $P_{n}(k)\cdot f$ a {\it semisimple orbit}) if $g$ is a semisimple element 
in $\mathfrak{l}_{n-j}(k)^{\ast}$ (or $L_{n-j}(k)\cdot g$ is a semisimple orbit). The latter means
$\Stab_{L_{n-j}(k)}(g)$ is an a reductive (algebraic) subgroup of $L_{n-j}(k)\cong\GL(n\!-\!1\!-\!j,k)$. 
In this case we call the $(n\!-\!1\!-\!j)$ eigenvalues of $\pr^{-1}(g)\in\mathfrak{gl}(n\!-\!1\!-\!j,k)$ 
the eigenvalues of $f$.

In the case of $j=n-1$, we have $\Stab_{P_{n}(k)}(f)=\Stab_{L_{n-j}(k)}(g)=\{1\}$, this gives an open orbit
in $\mathfrak{p}_{n}(k)^{\ast}$. In the case of $j<n-1$, we have $\dim\Stab_{P_{n}(k)}(f)=
\dim\Stab_{L_{n-j}(k)}(g)\geq n-1-j\geq 1$, none of these orbits is open. Again, by the proof of Lemma
\ref{L:Porbit-inductive} and Theorem \ref{T:Porbit}, we see that the complement of the above open orbit in
$\mathfrak{p}_{n}(k)^{\ast}$ is a closed subset of codimension one. This shows the following.

\begin{proposition}\label{P:openOrbit}
There is a unique open $P_{n}(k)$-coadjoint orbit in $\mathfrak{p}_{n}(k)^{\ast}$, which we denote by 
$\mathcal{O}_{0}$. It is dense and its complement is a codimension one closed subset. For any $f\in
\mathcal{O}_{0}$, we have $\Stab_{P_{n}(k)}(f)=1$.

In particular , $\mathcal{O}_{0}$ is the only strongly regular coadjoint orbit of $P_{n}(k)$ in
$\mathfrak{p}_{n}(k)^{\ast}$. Consequently, $P_{n}(k)$ has one and only one discrete series representation, 
which we denote by $\tau$. Moreover $\tau$ is attached to $\mathcal{O}_{0}$, and $L^{2}(P_{n}(k))\cong\tau
\otimes\tau$. 
\end{proposition}

We give some more precise information about the unique open $P_{n}(k)$-orbit in $\mathfrak{p}_{n}(k)^{\ast}$ 
below. 

\begin{example}\label{E:openOrbit1}
Let $$\xi_{n}=\left(\begin{array}{ccccc}0&0&\ldots&0&0\\1&0&\ldots&0&0\\\vdots&\vdots&\ddots&\vdots&\vdots
\\0&0&\ldots&0&0\\0&0&\ldots&1&0\\\end{array}\right),$$ and $f_{n}=\pr'(\xi_{n})\in\mathfrak{p}_{n}(k)^{\ast}$.
Write $h$ for the projection of $f_{n}$ to $\mathfrak{n}_{n}(k)^{\ast}$. By calculation one sees that 
$L_{n}(k)^{h}=P_{n-1}(k)$, and the element $f_{n}+\ad(\mathfrak{n}_{n}(k))\overline{h}\in
(\mathfrak{l}_{n}(k)^{h})^{\ast}$ is equal to $\pr'(\xi_{n-1})\in\mathfrak{p}_{n-1}(k)^{\ast}=
(\mathfrak{l}_{n}(k)^{h})^{\ast}$. Taking induction, one sees that $f_{n}$ represents the unique open 
$P_{n}(k)$ orbit in $\mathfrak{p}_{n}(k)^{\ast}$.
\end{example}

\begin{example}\label{E:openOrbit2}
Let $k=\mathbb{C}$. Let $a_1,\dots,a_{n-1},b_1,\dots,b_{n-1}\in\mathbb{C}$ such that $a_{i}\neq a_{j}$
($\forall i,j$, $1\leq i<j\leq n-1$) and $b_{j}\neq 0$ ($\forall j$, $1\leq j\leq n-1$). Set $$\xi'_{n}
=\left(\begin{array}{ccccc}a_1&0&\ldots&0&0\\0&a_2&\ldots&0&0\\\vdots&\vdots&\ddots&\vdots&\vdots\\0&0&
\ldots&a_{n-1}&0\\b_1&b_2&\ldots&b_{n-1}&0\\\end{array}\right)$$ and $f'_{n}=\pr'(\xi_{n})\in
\mathfrak{p}_{n}(k)^{\ast}$. We show that $f'_{n}$ represents the unique open $P_{n}(k)$ orbit in
$\mathfrak{p}_{n}(k)^{\ast}$.

Suppose $$g=\left(\begin{array}{cc}B&\beta\\0_{1\times(n-1)}&1\\\end{array}\right)\in\Stab_{P_{n}}
(f'_{n}),$$ where $B\in\GL(n-1,\mathbb{C})$ and $\beta\in\mathbb{C}^{n-1}$. Write $A=\diag\{a_1,\dots,
a_{n-1}\}$ and $\alpha=(b_1,\dots,b_{n-1})^{t}$. Then, $$\xi'_{n}=\left(\begin{array}{cc}A&0_{(n-1)
\times 1}\\\alpha^{t}&0\\\end{array}\right).$$ Due to $g\cdot f'_{n}=g\pr'(\xi'_{n})=\pr'(g\xi'_{n}g^{-1})$,
we see that $$g\in\Stab_{P_{n}}(f_{n})\Leftrightarrow g\xi'_{n}g^{-1}-\xi'_{n}\in\mathfrak{n}'_{n},$$
where \begin{eqnarray*}&&\mathfrak{n}'_{n}\\=&&\{\eta\in\mathfrak{gl}(n,\mathbb{C}):\tr(\eta\theta)=0,
\forall\theta\in\mathfrak{p}_{n}\}\\=&&\{\left(\begin{array}{cc}0_{n-1}&\alpha\\
0_{1\times(n-1)}&t\\\end{array}\right):\alpha\in\mathbb{C}^{n-1},t\in\mathbb{C}\}.\end{eqnarray*}
By calculation $$g\xi'_{n}g^{-1}= \left(\begin{array}{cc}BAB^{-1}+\beta\alpha^{t}B^{-1}&-BAB^{-1}\beta
-\beta\alpha^{t}B^{-1}\beta\\\alpha^{t}B^{-1}&-\alpha^{t}B^{-1}\beta\\\end{array}\right).$$ Thus,
$$g\in\Stab_{P_{n}}(f_{n})\Leftrightarrow BAB^{-1}+\beta\alpha^{t}B^{-1}=A\textrm{ and }\alpha^{t}B^{-1}
=\alpha^{t}.$$ This is equivalent to: $$AB-BA=\beta\alpha^{t}\textrm{ and }\alpha^{t}B=\alpha^{t}.$$
Write $$B=(x_{i,j})_{(n-1)\times(n-1)}\textrm{ and }\beta^{t}=(y_{j})_{1\leq j\leq(n-1)}.$$ From
$AB-BA=\beta\alpha^{t}$, we get $$(a_{i}-a_{j})x_{i,j}=b_{j}y_{i},\forall i,j,1\leq i,j\leq n-1.$$
Let $i=j$, from $b_{j}\neq 0$ we get $y_{j}=0$. Thus, $\beta=0$. Let $i\neq j$, from $y_{i}=0$ we get
$x_{i,j}=0$. Thus, $B$ is a diagonal matrix. From $\alpha^{t}B=\alpha^{t}$ and all entries of $\alpha$
are not equal to $0$, we get $B=I_{n-1}$. Therefore, $\Stab_{P_{n}}(f'_{n})=1$.
\end{example}

Remark:  After we finished this article, Prof. Ra\"{i}s informed us that some results in this sub-section were also obtained in his article \cite{Rais}. Especially in his article, the existence of open $P$-coadjoint orbit was established, and a representative element of the open $P$-orbit was given. However, our method is different from his, and our results are more explicit. Moreover, our method is useful for us to interpret Sahi's result in the framework of Duflo's orbit method in next section (see Section \ref{SS: Duflo-Sahi}).

\subsection{The moment map in the $\GL(n,\mathbb{C})$ case}\label{SS:GL(C)}

Let $G=\GL(n,\mathbb{C})$, and $$P=P_{n}(\mathbb{C})=\{\left(\begin{array}{cc}A&\alpha\\0_{1\times(n-1)}
&1\\\end{array}\right):A\in\GL(n-1,\mathbb{C}),\alpha\in\mathbb{C}^{n-1}\}.$$ Let $T$ be the maximal
torus of $G$ consisting of all diagonal matrices in $G$. Let $\vec{a}=(a_1,..., a_n)\in\mathbb{C}^{n}$
with $a_{i}\neq a_{j}$ ($1\leq i<j\leq n$). Write $$\xi=\xi_{\vec{a}}=\diag\{a_1,..., a_n\}\in\mathfrak{t}
\subset\mathfrak{g},$$ and $$f=f_{\vec{a}}=\pr(\xi)\in\mathfrak{g}^{\ast}.$$ Then $f$ is a regular
semisimple element in $\mathfrak{g}^{\ast}$, and any regular semisimple orbit $\mathcal{O}\subset
\mathfrak{g}^{\ast}$ is of the form $\mathcal{O}=\mathcal{O}_{f}=G\cdot f$ with $f=f_{\vec{a}}$ as
above.

From now on, we fix $\vec{a}$, $\xi$ and $f$ in this subsection.

\begin{lemma}\label{L:double1}
There are exactly $2^{n}\!-\!1$ $P$ orbits in $\mathcal{O}_{f}$. One of them is Zariski open and dense,
and the union of the rest is a codimension one Zariski closed subset.
\end {lemma}

\begin{proof}
With the assumption, $\Stab_{G}(f)=T$. Thus, $P\backslash\mathcal{O}_{f}=P\backslash G/T$. It is clear
that the map $$PgT\mapsto Tg^{-1}P$$ gives a bijection $P\backslash G/T\cong T\backslash G/P.$ Consider
the transitive $G$-action on $\mathbb{C}^{n}\!-\!\{0\}$, $$g\cdot(x_1,..., x_n)^{t}=(g^{-1})^{t}(x_1,
\dots,x_n)^{t},\ \forall g\in G,\forall(x_1,\dots,x_{n})\in\mathbb{C}^n\!-\!\{0\}.$$ Let $v_0=(0,...,
0,1)^{t}$. Then, $\Stab_{G}(v_0)=P$. Thus, $G/P=\mathbb{C}^n-\{0\}$ and $T\backslash G/P=
T\backslash(\mathbb{C}^n-\{0\})$.

For any $I=\{i_1,...,i_k\}$ where $1\leq i_1<i_2<...<i_k\leq n$ and $1\leq k\leq n$, let $v_{I}=(x_1,
\dots,x_{n})^{t}\in\mathbb{C}^n-\{0\}$ be defined by $x_{i}=1$ if $i\in\{i_1,\dots,i_{k}\}$, and $x_{i}=0$
if $i\not\in\{i_1,\dots,i_{k}\}$. Note that $T$ acts on $\mathbb{C}^{n}-\{0\}$ through $$\diag\{\lambda_1,
\dots,\lambda_n\}\cdot(x_1,..., x_n)^{t}=(\lambda_1^{-1}x_1,..., \lambda_n^{-1}x_n)^{t}.$$ It is easy
to see that $\{v_{I}:\emptyset\neq I\subset\{1,...,n\}\}$ represent all different $T$ orbits in
$\mathbb{C}^{n}-\{0\}$. This shows the lemma.
\end{proof}

Write $I_0=\{1,2,...,n\}$. For any $\emptyset\neq I\subset I_0$, introduce a matrix $g_{I}\in G$. In the
case of $n\in I$, define $g_{I}$ by $g_{I}=\left(\begin{array}{cc}I_{n-1}&0_{(n-1)\times 1}\\\beta^{t}&
0\\\end{array}\right)$, where $\beta^{t}=(x_1,\dots,x_{n-1})$ with $x_{i}=1$ if $i\in I$, and $x_{i}=1$
if $i\not\in I$. In the case of $n\not\in I$, let $k\!=\!\max\{1\leq\!i\leq\!n\!-\!1\!:\!i\in I\}.$ Define
$g_{I}$ by $$g_{I}=\left(\begin{array}{cccc}I_{k-1}&0_{(k-1)\times 1}&0_{(k-1)\times(n-1-k)}&0_{(k-1)
\times 1}\\0_{1\times(k-1)} &0&0_{1\times(n-1-k)}&1\\0_{(n\!-\!1\!-\!k)\times(k-1)}&0_{(n\!-\!1\!-\!k)
\times 1}&I_{n\!-\!1\!-\!k}&0_{(n\!-\!1\!-\!k)\times 1}\\\beta'^{t}&1&0_{1\times(n-1-k)}&0\\\end{array}
\right),$$ where $\beta'^{t}=(x_1,\dots,x_{k-1})$ with $x_{i}=1$ if $i\in I$, and $x_{i}=0$ if $i\not
\in I$.

\begin{proposition}\label{P:double2}
$\{g_{I}\cdot f:\emptyset\neq I\subset I_0\}$ represent all different $P$ orbits in $\mathcal{O}_{f}$.
Among these orbits, $Pg_{I_{0}}\cdot f$ is a Zarisk open and dense $P$ orbit and its complement is a
Zariski closed subset of codimension one.
\end{proposition}
\begin{proof}
One can show that $g_{I}^{-1}\cdot v_0=v_{I}$ for any $I$. By the proof of Lemma \ref{L:double1}, $\{g_{I}
\cdot f:\emptyset\neq I\subset I_0\}$ represent all different $P$ orbits in $\mathcal{O}_{f}$.

It is clear that $Tg_{I_{0}}^{-1}\cdot v_0$ is a Zarisk open and dense subset in $\mathbb{C}^{n}-\{0\}$,
and its complement is a Zariski closed subset of codimension one. Thus, $Pg_{I_{0}}\cdot f$ is a Zarisk
open and dense $P$ orbit and its complement is a Zariski closed subset of codimension one.
\end{proof}

\begin{lemma}\label{L:moment1}
For any $\emptyset\neq I\subset I_0$, the element $p(g_{I}\cdot f)$ is semisimple, its depth is equal to 
$\#I$ and its eigenvalues are $\{a_{i}: i\in I_{0}-I\}$.
\end{lemma}
\begin{proof}
We have $$p(g_{I}\cdot f)=p(g_{I}\cdot\pr(\xi))=p(\pr(g_{I}\cdot\xi))=\pr'(g_{I}\cdot\xi).$$

In the case of $n\in I$, by calculation we have $\pr'(g_{I}\cdot\xi)=\pr'(\xi_{I})$, where $$\xi_{I}=  \left(\begin{array}{cc}\diag\{a_1,\dots,a_{n-1}\}&0_{(n-1)\times 1}\\\beta^{t}&0\end{array}\right))$$ 
with $\beta^{t}=(x_1,\dots,x_{n-1})$, $x_{i}=a_{i}-a_{n}$ if $i\in I$, and $x_{i}=0$ if $i\not\in I$. 
Separating $I_0=\{1,2,\dots,n\}$ into the disjoint union of two subsets, $I$ and $I_0-I$, we see that 
$\xi_{I}$ is a block diagonal matrix. The one with rows and columns indexed by $I$ is of the form in 
Example \ref{E:openOrbit2}; the one with rows and columns indexed by $I_0-I$ is a diagonal matrix with 
eigenvalues $\{a_{i}: i\in I_{0}-I\}$.

We know that there is a unique open $P$ orbit in $\mathfrak{p}^{\ast}$, and any matrix of the form in
Example \ref{E:openOrbit1} or \ref{E:openOrbit2} is in this orbit. From this, substituting $\xi_{I}$
by a $P$ conjugate matrix $\xi'_{I}$, we could make $\xi'_{I}$ still a block diagonal matrix with two
blocks indexed by $I$ and $I_0-I$ respectively, with the part indexed by the set $I$ of the form in
Example \ref{E:openOrbit1} (with degree of $\#I$, instead of $n$), and the part indexed by $I_0-I$
a digonal matrix with eigenvalues $\{a_{i}: i\in I_{0}-I\}$. Applying the reduction \ref{T:Porbit}
$\#I-1$ times, we arrive at a diagonal matrix in $\mathfrak{l}_{n\!+\!1\!-\!\#I}(\mathbb{C})$ with 
eigenvalues $\{a_{i}: i\in I_{0}-I\}$. That just means, $g_{I}\cdot f$ is semisimple, with depth 
equal to $\#I$, and its eigenvalues are $\{a_{i}: i\in I_{0}-I\}$. 

In the case of $n\in I$, let $k=\max\{i: i\in I\}$. By calculation we have $g_{I}\cdot f=
\pr'(g_{I}\cdot\xi)=\pr'(\xi_{I})$, where $$\xi_{I}=\left(\begin{array}{cc}\diag\{a_1,\dots,a_{k-1},
a_{n},a_{k+1},\dots,a_{n-1}\}&0_{(n-1)\times 1}\\\beta^{t}&0\end{array}\right))$$ with $\beta^{t}=
(x_1,\dots,x_{k-1},0,\dots,0)$, where $x_{i}=a_{i}-a_{n}$ if $i\in I$, and $x_{i}=0$ if $i\not\in I$.
Write $$I'=(I-\{k\})\cup\{n\}.$$ Separating $I_0=\{1,2,\dots,n\}$ into the disjoint union of two
subsets, $I'$ and $I_0-I'$, we see that $\xi_{I}$ is a block diagonal matrix. The one with rows and
columns indexed by $I'$ is of the form in Example \ref{E:openOrbit2}; the one with rows and columns
indexed by $I_0-I'$ is a diagonal matrix with eigenvalues $\{a_{i}: i\in I_{0}-I\}$. The proof of
the rest is the same in the case of $n\in I$.
\end{proof}

\begin{theorem}\label{T:moment1}
The set $p(\mathcal{O}_{f})$ consists of exactly $2^{n}-1$ semisimple $P$ orbits, with the unique open $P$ 
coadjoint orbit in $\mathfrak{p}^{\ast}$ among them. Moreover,  we have:
\begin{itemize}
\item[(1),] the moment map $p:\mathcal{O}_{f}\rightarrow\mathfrak{p}^{\ast}$ is \emph{weakly proper} .
\item[(2),] the \emph{reduced space} of the unique open $P$ orbit in $\mathfrak{p}^{\ast}$ with respect to
the moment map $p$ is a single point.
\end{itemize}
\end{theorem}

\begin{proof}
By Lemma \ref{L:double1}, $\mathcal{O}_{f}$ is the union of exactly $2^{n}-1$ $P$ orbits. In Lemma
\ref{L:moment1}, we described the image of the moment map $p:\mathfrak{g}^{\ast}\rightarrow
\mathfrak{p}^{\ast}$ for each of these $P$ orbits. Particularly, we see that: each of them is semisimple,
and different $P$ orbits in $\mathcal{O}_{f}$ are mapped to different $P$ orbits in $\mathfrak{p}^{\ast}$
(due to $a_{i}\neq a_{j}$), and the unique open $P$ orbit in $\mathfrak{p}^{\ast}$ is among them (for
$I=I_0$). This shows the first statement of the theorem.

The second statement also follows from the description of the moment map.
\end{proof}

In Theorem \ref{T:moment1}, for any $\emptyset\neq I\subset I_0$, write $j=\#I-1$. Then, $$\Stab_{P}(g_{I}
\cdot f)=P\cap g_{I}Tg_{I}^{-1}=g_{I}(T\cap g_{I}^{-1}Pg_{I})g_{I}^{-1},$$ and $$T\cap g_{I}Pg_{I}^{-1}=
\Stab_{T}(g_{I}^{-1}\cdot v_0),$$ which is a torus isomorphic to $(\mathbb{C}^{\times})^{n-1-j}$. From 
the description of $p(g_{I}\cdot f)$ in Lemma \ref{L:moment1}, by Lemma \ref{L:Porbit-stabilizer} and 
Lemma \ref{L:Porbit-stabilizer2} we have $$\Stab_{P}(p(g_{I}\cdot f))\cong\mathbb{C}^{n-1-j}\rtimes
(\mathbb{C}^{\times})^{n-1-j}.$$

\subsection{The $\GL(n,\mathbb{R})$ case}\label{SS:GL(R)}

Now let $G=\GL(n,\mathbb{R})$, and $$P=P_{n}(\mathbb{R})=\{\left(\begin{array}{cc}A&\alpha\\
0_{1\times(n-1)}&1\\\end{array}\right):A\in\GL(n-1,\mathbb{R}),\alpha\in\mathbb{R}^{n-1}\}.$$
Let $0\leq k\leq\frac{n}{2}$. Let $z_{1},\dots,z_{n}$ be $n$ distinct complex numbers with
$z_{2j-1}=a_{j}+\mathbf{i}b_{j}$ ($1\leq j\leq k$), $z_{2j}=a_{j}-\mathbf{i}b_{j}$ ($1\leq j
\leq k$), and $z_{2k+j}=a_{2k+j}$ ($1\leq j\leq n-2k$), where $a_1,\dots,a_{k},a_{2k+1},\dots,
a_{n},b_1,\dots,b_{k}\in\mathbb{R}$ and $b_1\cdots b_{k}\neq 0$. Write $$\vec{z}=(z_1,...,z_{n})
\in\mathbb{C}^{n}.$$ Set $$J=\left(\begin{array}{cc}0&1\\-1&0\\\end{array}\right).$$ Write $$\xi
=\xi_{\vec{z}}=\diag\{a_1I_{2}+b_1J,\dots,a_{k}I_{2}+b_{k}J,a_{2k+1},\dots,a_{n}\}\in\mathfrak{g},$$
and $$f=f_{\vec{z}}=\pr(\xi)\in\mathfrak{g}^{\ast}.$$ Write \begin{eqnarray*}T_{k}&=&
\{\diag\{\lambda_1I_{2}\!+\!\mu_1J,\dots,\lambda_{k}I_{2}\!+\!\mu_{k}J,\lambda_{2k\!+\!1},\dots,
\lambda_{n}\}:\\&&\lambda_{1},\dots,\lambda_{k},\lambda_{2k\!+\!1},\dots,\lambda_{n},\mu_1,\dots,
\mu_{k}\in\mathbb{R},\\&&(\lambda_{1}^2\!+\!\mu_{1}^{2})\cdots(\lambda_{k}^{2}\!+\!\mu_{k}^{2})
\lambda_{2k\!+\!1}\cdots\lambda_{n}\neq 0\},\end{eqnarray*} which is a maximal torus in $G$.
It is clear that $\Stab_{G}(f)=T_{k}$. Thus, $f$ is a regular semisimple element in
$\mathfrak{g}^{\ast}$. Any regular semisimple orbit $\mathcal{O}\subset\mathfrak{g}^{\ast}$ is
of the form $\mathcal{O}=\mathcal{O}_{f}=G\cdot f$ for some $0\leq k\leq\frac{n}{2}$, $a_1,\dots,
a_{k},a_{2k+1},\dots,a_{n},b_1,\dots,b_{k}\in\mathbb{R}$ with $b_1\cdots b_{k}\neq 0$, and
$f=f_{\vec{z}}$ as above. From now on, we fix $\vec{z}$, $\xi$ and $f$ in this subsection.

\begin{lemma}\label{L:double3}
There are exactly $2^{n-k}\!-\!1$ $P$ orbits in $\mathcal{O}_{f}$. One of them is open and dense,
and the union of the rest is a codimension one (if $k<\frac{n}{2}$) or two (if $k=\frac{n}{2}$)
closed subset.
\end {lemma}

\begin{proof}
Let $G$ act on $\mathbb{R}^{n}-\{0\}$ though $$g\cdot(x_1,\dots,x_{n})^{t}=(g^{-1})^{t}(x_1,\dots,
x_{n})^{t},\ \forall g\in G,\forall (x_1,\dots,x_{n})^{t}\in\mathbb{R}^{n}-\{0\}.$$ Write $v_0=
(0,\dots,0,1)^{t}\in\mathbb{R}^{n}-\{0\}$. Then, $\Stab_{G}(v_0)=P$ and $\mathbb{R}^{n}-\{0\}=G/P$.
Similar as in the proof of Lemma \ref{L:moment1}, we have identifications $$P\backslash\mathcal{O}_{f}
=P\backslash G/T_{k}\cong T_{k}\backslash G/P=T_{k}\backslash(\mathbb{R}^{n}-\{0\}).$$

Write $I_{0}^{(1)}=\{1,\dots,k\}$, $I_{0}^{(2)}=\{2k+1,\dots,n\}$. For any $I_1\subset I_{0}^{(1)}$
and $I_2\subset I_{0}^{(2)}$ with $\#I_{1}+\#I_{2}>0$, let $v_{I_1,I_2}=(x_1,\dots,x_{n})\in
\mathbb{R}^n-\{0\}$ be defined by $(x_{2i-1},x_{2i})=(0,1)$ if $i\in I_{1}$, $(x_{2i-1},x_{2i})=(0,0)$
if $i\in I_{0}^{(1)}-I_{1}$, $x_{i}=1$ if $i\in I_{2}$, and $x_{i}=0$ if $i\in I_{0}^{(2)}-I_{2}$.
Note that $T_{k}$ acts on $\mathbb{R}^{n}-\{0\}$ through \begin{eqnarray*}&&\diag\{\lambda_1I_{2}+
\mu_{1}J,\dots,\lambda_{k}I_{2}+\mu_{k}J,\lambda_{2k+1},\dots,\lambda_n\}\cdot(x_1,..., x_n)^{t}
\\&=&((\lambda_{1}^{2}\!+\!\mu_{1}^{2})^{-1}(x_1\lambda_1\!+\!y_{1}\mu_1),(\lambda_{1}^{2}\!+\!
\mu_{1}^{2})^{-1}(-x_1\mu_1\!+\!y_{1}\lambda_1),\dots,\\&&(\lambda_{k}^{2}\!+\!\mu_{k}^{2})^{-1}
(x_{k}\lambda_{k}\!+\!y_{k}\mu_{k}),(\lambda_{k}^{2}\!+\!\mu_{k}^{2})^{-1}(-x_{k}\mu_{k}\!+\!
y_{k}\lambda_{k}),\\&&\lambda_{2k\!+\!1}^{-1}x_{2k\!+\!1},\dots,\lambda_n^{-1}x_n)^{t}.
\end{eqnarray*} It is easy to see that $\{v_{I_1,I_2}:\emptyset\neq I\subset\{1,...,n\}\}$
represent all different $T_{k}$ orbits in $\mathbb{R}^{n}-\{0\}$. This shows the lemma.
\end{proof}

For $(I_1,I_2)$ as in the above proof, define a matrix $g_{I_1,I_2}\in G$. In the case of $n\in I_{2}$,
let $$g_{I_1,I_2}=\left(\begin{array}{cc}I_{n-1}&0_{(n-1)\times 1}\\\beta^{t}&1\\\end{array}\right),$$
where $\beta^{t}=(x_1,\dots,x_{n-1})$ with $(x_{2i-1},x_{2i})=(0,1)$ if $i\in I_1$, $(x_{2i-1},x_{2i})
=(0,0)$ if $i\in I_{0}^{(1)}-I_1$, $x_{i}=1$ if $i\in I_{2}$, and $x_{i}=0$ if $i\in I_{0}^{(2)}-I_2$.

In the case of $I_{2}\not=\emptyset$ and $n\in I_{0}^{(2)}-I_2$, set $k'\!=\!\max\{1\leq\!i\leq\!n\!-
\!1\!:\!i\in I_2\}.$ Let $$g_{I_1,I_2}=\left(\begin{array}{cccc}I_{k'-1}&0_{(k'-1)\times 1}&0_{(k'-1)
\times(n-1-k')}&0_{(k'-1)\times 1}\\0_{1\times(k'-1)} &0&0_{1\times(n-1-k')}&1\\0_{(n\!-\!1\!-\!k')
\times(k'-1)}&0_{(n\!-\!1\!-\!k')\times 1}&I_{n\!-\!1\!-\!k'}&0_{(n\!-\!1\!-\!k')\times 1}\\\beta'^{t}
&1&0_{1\times(n-1-k')}&0\\\end{array}\right),$$ where $\beta'^{t}=(x_1,\dots,x_{k'-1})$ with $(x_{2i-1},
x_{2i})=(0,1)$ if $i\in I_1$, $(x_{2i-1},x_{2i})=(0,0)$ if $i\in I_{0}^{(1)}-I_1$, $x_{i}=1$ if
$i\in I_{2}$, and $x_{i}=0$ if $i\not\in I_{0}^{(2)}-I_2$.

In the case of $I_{2}=\emptyset$ and $\frac{n}{2}\in I_1$, let $$g_{I_1,I_2}=\left(\begin{array}{cc}
I_{n-1}&0_{(n-1)\times 1}\\\beta^{t}&1\\\end{array}\right),$$ where $\beta^{t}=(x_1,\dots,x_{n-1})$
with $(x_{2i-1},x_{2i})=(0,1)$ if $i\in I_1$ ($1\leq i<\frac{n}{2}$), $(x_{2i-1},x_{2i})=(0,0)$ if
$i\in I_{0}^{(1)}-I_1$ ($1\leq i<\frac{n}{2}$), and $x_{n-1}=0$.

In the case of $I_{2}=\emptyset$ and $\frac{n}{2}\not\in I_1$, set $k'\!=\!\max\{1\leq\!i<\!
\frac{n}{2}\!:\!i\in I_1\}.$ Let $$g_{I_1,I_2}=\left(\begin{array}{cccc}I_{2k'-1}&0_{(2k'-1)\times 1}
&0_{(2k'-1)\times(n-1-2k')}&0_{(2k'-1)\times 1}\\0_{1\times(2k'-1)} &0&0_{1\times(n-1-2k')}&1\\
0_{(n\!-\!1\!-\!2k')\times(2k'-1)}&0_{(n\!-\!1\!-\!2k')\times 1}&I_{n\!-\!1\!-\!2k'}&0_{(n\!-
\!1\!-\!2k')\times 1}\\\beta'^{t}&1&0_{1\times(n-1-2k')}&0\\\end{array}\right),$$ where $\beta'^{t}
=(x_1,\dots,x_{2k'-1})$ with $(x_{2i-1},x_{2i})=(0,1)$ if $i\in I_1$, $(x_{2i-1},x_{2i})=(0,0)$ if
$i\in I_{0}^{(1)}-I_1$, and $x_{2k'-1}=0$.

\begin{proposition}\label{P:double4}
$\{g_{I_1,I_2}\cdot f:(\emptyset,\emptyset)\neq (I_1,I_2)\subset(I_{0}^{(1)},I_{0}^{(2)})\}$ represent
all different $P$ orbits in $\mathcal{O}_{f}$. Among these orbits, $Pg_{I_{0}^{(1)},I_{0}^{(2)}}
\cdot f$ is an open and dense $P$ orbit, and its complement is a closed subset of codimension one
(in the case of $k\neq\frac{n}{2}$) or two (in the case of $k=\frac{n}{2}$).
\end{proposition}

\begin{proof}
One can show that $g_{I_1,I_2}^{-1}\cdot v_0=v_{I_1,I_2}$ for any $(I_1,I_2)$. By the proof of Lemma
\ref{L:double1}, this indicates that $$\{g_{I_1,I_2}\cdot f:(\emptyset,\emptyset)\neq (I_1,I_2)
\subset(I_{0}^{(1)},I_{0}^{(2)})\}$$ represent all different $P$ orbits in $\mathcal{O}_{f}$.

It is clear that $Tg_{I_{0}^{(1)},I_{0}^{(2)}}^{-1}\cdot v_0$ is an open and dense subset in
$\mathbb{R}^{n}-\{0\}$, and its complement is a Zariski closed subset of codimension one (in the
case of $k\neq\frac{n}{2}$) or two (in the case of $k=\frac{n}{2}$). Thus, $Pg_{I_{0}^{(1)},
I_{0}^{(2)}}\cdot f$ is an open and dense $P$ orbit, and its complement is a closed subset of
codimension one (in the case of $k\neq\frac{n}{2}$) or two (in the case of $k=\frac{n}{2}$).
\end{proof}

\begin{lemma}\label{L:moment2}
For any $(\emptyset,\emptyset)\neq (I_1,I_2)\subset(I_{0}^{(1)},I_{0}^{(2)})$, the element $p(g_{I_1,
I_2}\cdot f)$ is semisimple, its depth is equal to $2\#I_1+\#I_2$, and its eigenvalues are
$\{z_{2i-1},z_{2i},z_{j}:i\in I_{0}^{(1)}-I_{1},j\in I_{0}^{(2)}-I_2\}$.
\end{lemma}

\begin{proof}
Regard $\mathfrak{p}_{n}(\mathbb{R})$ as a real form of $\mathfrak{p}_{n}(\mathbb{C})$. Then, $\mathfrak{p}_{n}
(\mathbb{R})^{\ast}$ is naturally contained in $\mathfrak{p}_{n}(\mathbb{C})^{\ast}$ as a real form
of it. By the proof of Theorem \ref{T:Porbit}, we see that this imbedding does not change the depth
and eigenvalues. It is convenient to do conjugation regarding $P_{n}(\mathbb{C})$ to see the depth
and eigenvalues of $p(g_{I_1,I_2}\cdot f)$, and hence shows the lemma.
\end{proof}

The following theorem could be shown similarly as Theorem \ref{T:moment1}.
\begin{theorem}\label{T:moment2}
The set $p(\mathcal{O}_{f})$ consists of exactly $2^{n-k}-1$ semisimple $P$ orbits, with the unique
open $P$ orbit in $\mathfrak{p}^{\ast}$ among them. Moreover,  we have:
\begin{itemize}
\item[(1),] the moment map $p:\mathcal{O}_{f}\rightarrow\mathfrak{p}^{\ast}$ is \emph{weakly proper}.
\item[(2),] the \emph{reduced space} of the unique open $P$ orbit in $\mathfrak{p}^{\ast}$ with respect
to the moment map $p$ is a single point.
\end{itemize}
\end{theorem}

Analogous to the $\GL(n,\mathbb{C})$ case, in Theorem \ref{T:moment2} we write $$j=2\#I_1+\#I_{2}-1\in
[0,n-1]$$ for any $(\emptyset,\emptyset)\neq (I_1,I_2)\subset(I_{0}^{(1)},I_{0}^{(2)})$. Then, with 
Lemma \ref{L:moment2} one can show that $$\Stab_{P}(g_{I_1,I_2}\cdot f)\cong\U(1)^{k-\!\#I_{1}}\!\times
\!(\mathbb{R}^{\times})^{n\!-\!k\!-\!\#I_1\!-\!\#I_2}.$$ On the other hand, Lemma \ref{L:Porbit-stabilizer} 
and Lemma \ref{L:Porbit-stabilizer2} indicate that $$\Stab_{P}(p(g_{I_1,I_2}\cdot f))\cong\mathbb{R}^{n\!-
\!1\!-\!j}\!\rtimes(\U(1)^{k\!-\!\#I_1}\!\times\!(\mathbb{R}^{\times})^{n\!-\!k-\#I_1\!-\!\#I_2}).$$

\section{Kirillov's conjecture and orbit method}\label{Kirillov}

\subsection{General setting}

Let $G=\GL(n, \mathbb{K})$ with Lie algebra $\mathfrak{g}$. Let $$P=\{\left(\begin{array}
{cc}A&\alpha \\0&1\end{array}\right): A\in \GL(n_1,\mathbb{K}),\alpha\in\mathbb{K}^{n-1}\}$$
with Lie algebra $\mathfrak{p}$. Identify $\mathfrak{g}$ with its algebraic dual $\mathfrak{g}^{\ast}$ 
via the trace function $\text{tr} $. Let $T$ be the diagonal torus of $G$ with Lie algebra $\mathfrak{t}$.

Recall that by  Proposition \ref{P:openOrbit}, there is one and only one open $P$-coadjoint 
orbit in $\mathfrak{p}^*$. For now on, we denote by $\Omega$, the unique open $P$-coadjoint orbit in
$\mathfrak{p}^{\ast}$; and by $\tau$, the (discrete series) representation of $P$ attached to $\Omega$.

Let $\pi$ be a tempered representation (with regular infinitesimal character) of $G$.
Then $\pi$ is associated to a strongly regular coadjoint orbit $\mathcal{O}_{\pi}$ (in the sense of Duflo). 
Notice that since $G$ is reductive, all regular $G$-coadjoint orbits are strongly regular.  Let $\text{p}:
\mathcal{O}_{\pi}\rightarrow\mathfrak{p}^{\ast}$ be the moment map of the $P$-Hamiltonian space 
$\mathcal{O}_{\pi}$. Then the following theorem serves not only as a geometric interpretation of Kirillov's 
conjecture for tempered representations, but as a generalization of Duflo's conjecture (to tempered 
representations):

\begin{theorem}\label{main theorem}

There are only finitely many $P$-orbits in $\text{p}(\mathcal{O}_{\pi})$, and the
unique open $P$- coadjoint orbit $\Omega$ is contained in $\text{p}(\mathcal{O}_{\pi})$. Moreover, we have

\begin{itemize}
\item[i)] The moment map $\text{p}: \mathcal{O}_{\pi}\rightarrow\mathfrak{p}^{\ast}$ is
\emph{weakly proper} .
\item[ii)] The restriction of $\pi$ to $P$, $\pi\vert_{P}$ (which is irreducible) is attached
to $\Omega$ (in the sense of Duflo). In other words, $\pi\vert_{P}=\tau$.
\item[iii)] The \emph{reduced space} of $\Omega$  (with respect to the moment map $\text{p}$)
is a single point.
\end{itemize}
\end {theorem}

Remarks: \begin{enumerate}
\item The weak properness in the above theorem "predicts" the $P$-admissibility (in the sense
of Kobayashi) of $\pi$.
\item Since there is  one and only one open $P$-coadjoint orbit (namely $\Omega$) in
$\text{p}(\mathcal{O}_{\pi})$,  it is also the unique strongly regular $P$-coadjoint orbit
contained in in $\text{p}(\mathcal{O}_{\pi})$.  This geometric fact "implies" that only the
irreducible representation of $P$ attached to  $\Omega$, namely  $\tau$, can appear in the
restriction $\pi\vert_{P}$.
\item The fact that reduced space of  $\Omega$ is a single point "signifies" that the multiplicity
of $\tau$ in $\pi\vert_{P}$ is no more than one (so is exactly one in our context).
\end{enumerate}

These three points together mean exactly $\pi\vert_{P}=\tau$. \\

We already proved, in previous section (Theorem \ref{T:moment1} and Theorem \ref{T:moment2}), 
that except  for $\text{ii)}$, the Theorem  \ref{main theorem} is true for all regular
$G$-coadjoint orbits $\mathcal{O}\subset  \mathfrak{g}^*$ (which are not necessarily attached
to $G$-representations).  Then we only need to treat $\text{ii)}$ of the Theorem  \ref{main theorem}.

\subsection{Sahi's results on Kirillov's conjecture} 

A first success for Kirillov's conjecture was accomplished by Sahi. He established (in \cite{Sahi}) 
Kirillov's conjecture completely for $\GL(n,\mathbb{C})$ and partially for $\GL(n,\mathbb{R})$. 
In particular he proved the conjecture for all tempered representations $\pi$ of any $\GL(n,\mathbb{K})$ 
(for $\mathbb{K}=\mathbb{C}$ or $\mathbb{R}$). In fact, for $\pi$ tempered, he determined more or less 
explicitly  $\pi\vert_{P}$ as an irreducible representation of $P$ based on two functors "$I$" and 
"$E$". Especially, it turns out that for any tempered representation $\pi$ of $\GL(n,\mathbb{K})$, 
$\pi\vert_{P}$ is a same irreducible representation of $P$,  which is denoted  by "$I^{n-1}E 1$" 
in Sahi's article. Let us briefly explain the construction of $I^{n-1}E 1$. For more details, 
the reader is referred to \cite{Sahi}.

So write $G_n:=\GL(n,\mathbb{K})$ and $P_n$ the subgroup in question.  We have the two facts:

\begin{itemize}
\item[1)] $P_n \cong \mathbb{K}^{n-1}\rtimes G_{n-1}$.
\item[2)] $ G_{n-1}$ has exactly two orbits in $ (\mathbb{K}^{n-1})^*$: $\{0\}$ and $(\mathbb{K}^{n-1})^{\ast}
\setminus\{0\}$. Moreover if we fix a character $\xi\in (\mathbb{K}^{n-1})^*\setminus \{0\}$ by $\xi((x_1,\dots,
x_{n-1}))=x_{n-1}$, then $\text{Stab}_{ G_{n-1}}(\xi)\cong P_{n-1} $.
\end{itemize}

Then we deduce by Mackey's classic theory that each irreducible unitary representation of 
$P_n$ is obtained in one of the two ways as follows:

\begin{itemize}
\item[a)] by trivially extending an irreducible unitary representation of $G_{n-1}$.
\item[b)] by extending an irreducible unitary representation of $P_{n-1}$ to $\mathbb{K}^{n-1}
\rtimes P_{n-1}$ by the a character $\xi$ and then inducing to $P_{n}$.
\end{itemize}

We use $E$ and $I$ for the above constructions a) and b) respectively. We can actually check that they are functors. Then we have $\widehat{P_n}=E(\widehat{G_{n-1}}) \bigsqcup I(\widehat{P_{n-1}})$. Moreover, using the convention that $P_1=G_0=$ the trivial group, we have the following fact: \\

Each irreducible unitary representation $\rho$ of $P_n$ is of the form: $ \rho=I^{k-1}E \sigma$ 
for some integer $k\geq 1$ and $ \sigma \in \hat{G_{n-k}}$. Moreover, $k$ and $\sigma$ are uniquely 
determined. The integer $k$ is called by Sahi the \emph{depth} of $\rho$.

We have the following theorem due to Sahi:

\begin{theorem}\label{T: Sahi}

 Let $\pi$ be a tempered representation of $G_n$. Then we have $\pi\vert_{P_n}=I^{n-1}E 1$.

\end{theorem}

So in order to prove ii) of the Theorem \ref{main theorem}, we only need to prove that $I^{n-1}E 1$ 
is attached to the unique open $P_n$-coadjoint orbit $\Omega$, namely $I^{n-1}E 1=\tau$. We will 
show it in the next subsection.

\subsection{Construction of $\tau$ in the framework of Duflo's orbit method}\label{SS: Duflo-Sahi}

As we mentioned, for any almost algebraic real groups, Duflo associated admissible and well 
polarizable coadjoint orbits to irreducible unitary representations. However, the general 
construction of irreducible representations attached to given orbits is carried out in an 
indirect and inductive manner, and involves some non-trivial ingredients (e.g., some "metaplectic" 
two-fold coverings, see  \cite{Duflo1}). Nevertheless, in our context, the construction of $\tau$, 
namely the representation attached to the unique open $P_n$-orbit $\Omega$, is quite transparent 
(though still by induction) in the framework of Duflo's theory, and it coincides with the classic 
Mackey theory.

Now Let us explain how to construct $\tau$. We retain the notation in the preceding subsection. 
Recall that $ P_n \cong \mathbb{K}^{n-1}\rtimes G_{n-1}$ and $\mathfrak{p}_{n}\cong
\mathbb{K}^{n-1} +  \mathfrak{g}_{n-1}$, with $\mathbb{K}^{n-1}$ the (abelian) nilradical of 
$ \mathfrak{p}_n $.

Firstly, we will choose an element $f\in \Omega$, it is known that  the construction doesn't depend 
on the choice of $f$. However, in order to be adapted to the construction of $I^{n-1}E 1$ by Sahi, 
we choose a $f\in \Omega$ such that $f\vert_{\mathbb{K}^{n-1}}=\xi$.  According to the proof of Lemma 
\ref{L:Porbit-inductive},  we know such $f$ exists, and  $\text{Stab}_{ P_{n}}(\xi)\cong
\mathbb{K}^{n-1}\rtimes P_{n-1} $. Note that in our case, the nilradical $\mathbb{K}^{n-1}$ 
(of $\mathfrak{p}_n$) is contained in $\text{Stab}_{ P_{n}}(\xi)$ (which implies that 
$\text{Stab}_{ P_{n}}(\xi).\mathbb{K}^{n-1}=\text{Stab}_{ P_{n}}(\xi)$). Then in the framework of 
Duflo's theory, we have

$$\tau=\text{Ind}^{P_{n}}_{ \mathbb{K}^{n-1}\rtimes P_{n-1}}(\xi \otimes \tilde{\tau}).$$

Here $\tilde{\tau} \in \widehat{P_{n-1}}$ is attached to the $P_{n-1}$-coadjoint orbit $P_{n-1}.\tilde{f}$, 
with $\tilde{f}=f\vert_{\mathfrak{p}_{n-1}}$. However, again by the proof of Lemma \ref{L:Porbit-inductive}, 
we see that $P_{n-1}.\tilde{f}$ is the unique open $P_{n-1}$-coadjoint orbit in $(\mathfrak{p}_{n-1})^*$. 
Then according to the construction of $I^{n-1}E 1$  and by a direct inductive argument, we obtain that 
$\tau=I^{n-1}E 1$. \\

Remarks: \begin{enumerate}
\item As  $I^{n-1}E 1$ is attached to  the open orbit $\Omega$, it is a discrete series.
\item Sahi defined the \emph{depth} for $ \rho=I^{k-1}E \sigma$.  In previous section, for any 
$P_n$-coadjoint orbit $\Xi$ , we also defined an integer called the \emph{depth} of $\Xi$. Actually 
we can check that the two notions are compatible, in the sense that if $\Xi$ is attached to an irreducible 
unitary representation $\rho$ (in the sense of Duflo), then the depth of  $\rho$  is that of $\Xi$.
\end{enumerate}

\end{document}